\documentclass[11pt]{article}
\usepackage[a4paper,top=28mm,bottom=28mm,inner=30mm,outer=30mm]{geometry}
\usepackage[hidelinks]{hyperref}

\usepackage{amsmath,amssymb}
\usepackage{amsthm}
\usepackage{enumitem}
\usepackage[toc]{appendix}
\usepackage{multicol}
\usepackage{comment}
\usepackage{hyperref}
\usepackage{url}
\usepackage{graphicx}
\usepackage{mathtools}

\usepackage{eucal, stmaryrd}
\usepackage{xcolor}
\usepackage{amsfonts} 
\usepackage{booktabs}
\usepackage{tikz}
\usepackage[margin=1cm]{caption}

\usepackage{pgffor}
\usepackage{bm}
\usepackage{mdframed}
\usepackage{hyperref}
\usepackage{pgfplots}
\pgfplotsset{compat=1.18}

\usetikzlibrary{tikzmark, angles, quotes, shapes.geometric}

\usepackage[backend=biber,
style=ieee,
sorting=nty
]{biblatex}
\usepackage{color}

\newcounter{BILPcounter}
\renewcommand{\theBILPcounter}{\arabic{BILPcounter}}

\newcounter{constraintcounter}
\renewcommand{\theconstraintcounter}{C\arabic{constraintcounter}}

\newenvironment{BILP}[1][]{%
    \refstepcounter{BILPcounter}%
    \begin{mdframed}[roundcorner=10pt]%
    \textbf{BILP~\theBILPcounter} #1\par\nobreak
    \setcounter{constraintcounter}{0}
}{%
    \end{mdframed}
}

\newenvironment{LP}[1][]{%
    \refstepcounter{BILPcounter}%
    \begin{mdframed}[roundcorner=10pt]%
    \textbf{LP~\theBILPcounter} #1\par\nobreak
    \setcounter{constraintcounter}{0}
}{%
    \end{mdframed}
}

\newcommand{\LPref}[1]{LP~\getrefnumber{#1}}
\newcommand{\BILPref}[1]{BILP~\getrefnumber{#1}}
\newcommand{\constrref}[1]{Constraint (\getrefnumber{#1})}\usepackage{adjustbox}

\newcommand{\constr}[1][]{%
  \refstepcounter{constraintcounter}%
  \ifx\\#1\\%
    \tag{\theconstraintcounter}%
  \else
    \label{#1}\tag{\theconstraintcounter}%
  \fi
}

\DeclareMathOperator{\spec}{spec}
\DeclareMathOperator{\ev}{ev}
\DeclareMathOperator{\tr}{tr}

\def\vec0{\mbox{\bf 0}}

\def\tr{\mathop{\rm tr }\nolimits}

\def\ev{\mathop{\rm ev}\nolimits}

\def\matrix0{{\mbox {\boldmath $O$}}}

\usepackage{listings}
\usepackage{tikz}
\tikzset{
every node/.style={circle, inner sep=2pt}
}
\usetikzlibrary{matrix,arrows,calc}

\newtheorem{theorem}{Theorem}
\newtheorem{lemma}[theorem]{Lemma}
\newtheorem{proposition}[theorem]{Proposition}
\newtheorem{corollary}[theorem]{Corollary}

\theoremstyle{definition}
\newtheorem{definition}[theorem]{Definition}
\newtheorem{observation}[theorem]{Observation}
\newtheorem{remark}[theorem]{Remark}
\newtheorem{example}[theorem]{Example}
\newtheorem{problem}{Problem}
\newtheorem{question}[theorem]{Question}

\def\1{{\bf 1}}

\usepackage{xcolor}

\title{Tales of Hoffman: from a distance}

\author{Aida Abiad \thanks{Department of Mathematics and Computer Science, Eindhoven University of Technology, The Netherlands,  Department of Mathematics and Data Science, Vrije Universiteit Brussel, Belgium \texttt{a.abiad.monge@tue.nl}}
\and
Jan Meeus \thanks{Department of Mathematics: Analysis, Logic and Discrete Mathematics, Ghent University, Belgium \texttt{jan.meeus@ugent.be}}
}

\date{}

\begin{document}
\maketitle

\begin{abstract}
Hoffman proved that a graph $G$ with adjacency eigenvalues $\lambda_1\geq \cdots \geq \lambda_n$ and chromatic number $\chi(G)$ satisfies
$\chi(G)\geq 1+\kappa,$
where $\kappa$ is the smallest integer such that 
$$\lambda_1+\sum_{i=1}^{\kappa}\lambda_{n+1-i}\leq 0.$$
We extend this eigenvalue bound to the distance-$k$ setting, and also show a strengthening of it by proving that it also lower bounds the corresponding quantum distance coloring graph parameter. 
The new bound depends on a degree-$k$ polynomial which can be chosen freely, so one needs to make a good choice of the polynomial to obtain as strong a bound as possible. We thus propose linear programming methods to optimize it. We also investigate the implications of the new bound for the quantum distance chromatic number, showing that it is sharp for some classes of graphs. Finally, we extend the Hoffman bound to the distance setting of the vector chromatic number. Our results extend and unify several previous bounds in the literature. \\

\end{abstract}

\section{Introduction}

In recent years, distance based colorings have gained significant attention in the literature, see e.g. \cite{agnarsson2003coloring, kral2004coloring, kramer2008survey, niranjan2019k, fertin2003acyclic}. Having been introduced in 1969 \cite{kramer1969farbungsproblem, kramer1969probleme}, a \emph{distance-$k$ coloring} is a natural extension of proper colorings, where vertices at distance at most $k$ should have distinct colors. Equivalently, one may consider them proper colorings of graph powers as studied in \cite{alon2002chromatic}. For a graph $G$, the \emph{$k$-th power graph} $G^k$ is defined as the graph with the same vertex set as $G$, but with edges between any two vertices at distance at most $k$. Clearly, a graph coloring is a proper coloring of $G^k$ if and only if it is distance-$k$ coloring of $G$. Another source of motivation for looking at distance-$k$ colorings comes from the fact that they have useful applications in coding theory, see, e.g. \cite{klotz2006distance, ostergaard2004hypercube, abiad2024eigenvalue}.

The \emph{distance-$k$ chromatic number} of a graph $G$ is the minimal number of colors $\chi_k(G)$ that a distance-$k$ coloring of $G$ must have. Equivalently, it is the chromatic number of the $k$-th \emph{power graph} of $G$. It is important to note that $\chi_k$ is an NP-hard parameter \cite{mccormick1983optimal}, thus much research focuses on finding good lower and upper bounds that can be efficiently computed (for instance, using eigenvalues, since then they can be computed in polynomial time). We should also note that algebraic and combinatorial properties of power graphs in general cannot be inferred from the properties of the original graph, and this in particular also applies to the eigenvalues of $G$ and $G^k$, which are in general not related (see \cite{XING201739,das2013laplacian, abiad2022optimization}). As such, eigenvalue bounds that have been proven to be famously useful for the chromatic number are of interest to study whether they hold in the more general distance-$k$ context. Note also that the usage of the classical bounds for the chromatic number on the $k$-th power graphs often yields suboptimal results when compared to the distance-$k$ specific bounds, so our main focus will be on deriving spectral bounds using the eigenvalues of the base graph $G$.

The most famous eigenvalue bound for the chromatic number of a graph is the celebrated \emph{Hoffman bound} \cite{hoffman2003eigenvalues}, which says that $\chi(G)\geq 1-\lambda_1/\lambda_n$, where $\lambda_1$ and $\lambda_n$ denote the largest and smallest adjacency eigenvalues of $G$ respectively. This bound was  extended to the distance-$k$ setting in \cite[Theorem 4.3]{abiad2022optimization} and to the distance-$k$ quantum setting in \cite[Theorem 14]{abiad2025eigenvalue} under the name of \emph{Hoffman-type bound}. In \cite{hoffman2003eigenvalues}, Hoffman also derived another bound on the chromatic number of a graph which is less popular but even stronger: 
\begin{theorem}[$\kappa$ bound  \cite{hoffman2003eigenvalues}]\label{thm:hoffmanbetterbound}
A connected graph $G$ on $n$ vertices with adjacency eigenvalues $\lambda_1\geq \cdots \geq \lambda_n$ and chromatic number $\chi(G)$ satisfies
\begin{equation}\label{bound:hoffmanrefinement}
    \chi(G)\geq 1+\kappa
\end{equation}
where $\kappa$ is the smallest integer such that 
$$\lambda_1+\sum_{i=1}^{\kappa}\lambda_{n+1-i}\leq 0.$$
\end{theorem}

The above bound can be seen as a refinement of the Hoffman bound. Heuristically, it is a much stronger bound than Hoffman's bound in terms of tightness frequency (see Table \ref{tab:comparisonchromaticbounds}), and it is easily shown that $1-\lambda_1/\lambda_n\leq 1+\kappa$ for any graph. Note also that the $\kappa$ bound may be formulated as the following inequality of eigenvalues:
$$
\lambda_1 + \sum_{i=1}^{\chi - 1}\lambda_{n-i+1}\leq 0.
$$
The above bound has been shown to also lower bound the quantum chromatic number of a graph \cite{wocjan2019spectral}, and it has also been used to show non-existence of certain classes of 5-chromatic strongly regular graphs \cite{fiala20065}. Other bounds on the chromatic number that encompass all eigenvalues of a graph have been studied in \cite{EW2013,EW2015,COUTINHO2019345}.

Motivated by the above $\kappa$ bound and its applications, we  extend it to the distance and quantum settings, and as a byproduct we obtain a unification of some existing bounds. In particular, we improve the Hoffman-type bound for the distance-$k$ quantum chromatic number from \cite[Theorem 14]{abiad2025eigenvalue}. On the other hand, we extend the $\kappa$ bound on the quantum chromatic number from \cite[Theorem 7]{wocjan2019spectral} to the distance-$k$ quantum setting. We achieve this by unifying the two known bounds into a \emph{unified $\kappa$ bound} for the distance-$k$ quantum chromatic number. 
The new bound depends on the choice of a polynomial of degree at most $k$, therefore we propose linear programming techniques to obtain the optimal bound for any specific graph. 

Next we provide infinite graph classes for which the new unified $\kappa$ bound is sharp (and for which the existing distance-$k$ Hoffman-type bound from \cite[Theorem 4.3]{abiad2022optimization} does not achieve equality). We also illustrate how to use the tightness of the new unified $\kappa$ bound to compute the value of the quantum distance chromatic number (and which could not have been obtained through the usage of known bounds). We do so by utilizing the fact that the distance-$k$ chromatic number is bounded below by the distance-$k$ quantum chromatic number, and equality of the new bound with the classical chromatic parameter therefore computes the quantum  chromatic parameter. Note that the latter is not known to be computable in general \cite{paulsen2016estimating}, thus obtaining new sharp bounds for this quantum chromatic parameter is of great importance.

Finally we investigate the \emph{distance-$k$ vector chromatic number} $\chi_{kv}$, which is the vector chromatic number of the $k$-th power graph. Motivated by the fact that the  Hoffman bound $1-\lambda_1/\lambda_n$ also holds for the vector chromatic number \cite{bilu2006tales}, we extend the distance-$k$ Hoffman-type bound \cite[Theorem 4.3]{abiad2022optimization} to the distance-$k$ vector chromatic number, and optimize it using linear programming methods. In doing so, we improve upon the Hoffman-type bound by obtaining a new bound which is proven to be at least as sharp, while it is also a lower bound for the distance-$k$ vector chromatic number.

\section{Preliminaries}

\subsection{Graph theory background}

For a graph $G=(V,E)$ on $n$ vertices we denote the \emph{adjacency matrix} of $G$ by $A(G)$ or $A=(a_{vw})\in\mathbb{R}^{n\times n}$, which is indexed by the vertices $V=V(G)$ of $G$. We denote the (adjacency) \emph{spectrum} by the multiset $\spec(G)=\{\lambda_1,\ldots,\lambda_n\}$ where $\lambda_1\geq \cdots \geq \lambda_n$ are the eigenvalues of $A(G)$. We denote the set of distinct eigenvalues of a graph $G$ as $\ev(G)=\{\theta_0,\ldots,\theta_d\}$ where $\theta_0>\ldots>\theta_d$ and $d\geq 0$. If $m_0,\ldots,m_d$ are the respective multiplicities of $\theta_0,\ldots,\theta_d$, then we may also write $\spec(G)=\{\theta_0^{[m_0]},\ldots,\theta_d^{[m_d]}\}$. 

We denote $[n]=\{1,2,\ldots,n\}$ for $n\in \mathbb{N}^+$. Given a polynomial $p\in \mathbb{R}_k[x]$ of degree at most $k$ we define the following parameters;
\begin{itemize}
    \item $W(p):=W(p(A)) = \max_{u\in V}\{p(A)_{uu}\}$;
    \item $w(p):=w(p(A)) = \min_{u\in V}\{p(A)_{uu}\}$;
    \item $\Lambda(p):=\Lambda(p(A)) = \max_{i\in [2,n]}\{p(\lambda_i)\}$;
    \item $\lambda(p):=\lambda(p(A)) = \min_{i\in [2,n]}\{p(\lambda_i)\}$.
\end{itemize}
Moreover, we denote $\Delta(G)$ and $\delta(G)$ for the maximum and minimum degree of a graph $G$, respectively. Additionally, the \emph{neighborhood}, i.e.\! the set of adjacent vertices in $G$, of a vertex $v\in V$ is denoted by $N_G(v)$. 

Let $A\in \mathbb{C}^{n\times n}$ and $i \in [n]$, then we denote $\lambda_i(A):=\lambda_i$ where $\lambda_1\geq \cdots\geq \lambda_n$ are the eigenvalues of $A$. Note that we avoid this notation if there are only a few matrices that we need to consider. Furthermore, we denote the Hermitian conjugate of a matrix $A$ by $A^{\dagger}$.

\begin{theorem}[Interlacing \cite{haemers1995interlacing}]\label{thm:interlacing}
    Let $S\in \mathbb{C}^{n\times m}$ be a matrix such that $S^{\dagger}S=I$. Let $A\in \mathbb{C}^{n\times n}$ be a hermitian matrix. Then the eigenvalues of $B:=S^{\dagger}AS$ interlace the eigenvalues of $A$, that is \begin{align*}
        \lambda_i(A)\geq \lambda_i(B)\geq\lambda_{n-m+i}(A) & \qquad\text{for } i\in \{0,\ldots,m\}.
    \end{align*} 
    If the interlacing is tight, i.e.\! there exists a $k\in\{0,\ldots,m\}$ such that
\begin{align*}
    \lambda_i&= \mu_i \qquad\text{for } 0 \leq i \leq k \text{ and}; \\
    \lambda_{n-m+i}&= \mu_i \qquad\text{for } k+1 \leq i \leq m,
\end{align*} 
then $SB=AS$.
\end{theorem}
\subsection{Quantum colorings}

Quantum colorings were first introduced as quantum strategies for nonlocal games based on pseudo-telepathy \cite{cleve2004consequences,galliard2002pseudo, cameron2006quantum}. However, there is an equivalent definition that does not require defining quantum strategies in their full generality \cite{cameron2006quantum}, which is purely combinatorial, and that has been used many times in the literature in order to obtain new results on quantum colorings \cite{manvcinska2018oddities, wocjan2018spectral, abiad2025eigenvalue,elphick2019spectral, godsil2024quantum}.

A \emph{distance-$k$ quantum $c$-coloring} is a quantum $c$-coloring of the power graph $G^k$, see the below definition for $k=1$. The \emph{distance-$k$ quantum chromatic number} is the quantum chromatic number of $G^k$, denoted $\chi_{kq}(G):=\chi_q(G^k)$. For completeness, we give a full definition of a distance-$k$ quantum $c$-coloring equivalent to the one with $G^k$.

An \emph{orthogonal projector} is a matrix $P\in \mathbb{C}^{n\times n}$ such that $P^2=P=P^{\dagger}$. A \emph{distance-$k$ quantum $c$-coloring} of a graph $G=(V,E)$ is a collection of orthogonal projectors $\{P_{v,i}\mid  v\in V, i\in [c]\}$ in $\mathbb{C}^{d\times d}$ such that \begin{itemize}
    \item for all vertices $v\in V$ \begin{align}
        \sum_{i\in [c]}P_{v,i}=I_d &\qquad\text{(completeness),}
    \end{align} 
    \item for all distinct vertices $v,w$ such that $d_G(v,w)\leq k$ and for all $i\in [c]$ \begin{align}
        P_{v,i}P_{w,i}=0 &\qquad\text{(orthogonality).}
    \end{align} 
\end{itemize}
The \emph{distance-$k$ quantum chromatic number} $\chi_{kq}(G)$ is the smallest $c$ for which the graph $G$ admits a distance-$k$ quantum $c$-coloring for some dimension $d>0$.

Note that, in the orthogonality condition we have the subtle restriction that $v,w$ must be \emph{distinct} in order to have orthogonality of $P_{v,i}$ and $P_{w,i}$, since otherwise we obtain $P_{v,i}^2=P_{v,i}=0$ for all $v\in V$ and $i\in [m]$. In the same spirit, notice that $v$ isn't adjacent to itself in the power graph $G^k$.

Furthermore, it is well-known that $\chi_q(G)\leq\chi(G)$, and this holds for power graphs in particular, so $\chi_{kq}(G)\leq \chi_k(G)$ for any graph $G$.

Let $P_k\in \mathbb{C}^{m\times m}$ for $k\in [c]$ be orthogonal projectors. Then, they form a \emph{resolution of the identity} if and only if $$
    \sum_{k\in [c]}P_k=I_m.
$$
The operation $\mathcal{D}$ defined by $$
    \mathcal{D}:\mathbb{C}^{m\times m}\rightarrow\mathbb{C}^{m\times m}:X \mapsto \mathcal{D}(X)=\sum_{k\in [c]}P_kXP_k
$$
is called \emph{pinching} if and only if the orthogonal projectors $P_k$ form a resolution of the identity, see e.g.\! \cite{wocjan2019spectral}. We say $\mathcal{D}$ \emph{annihilates} $X$ if $\mathcal{D}(X)=0$.

The following result will be useful in the proof of Theorem \ref{thm:quantumhoffmanbetterbound}.

\begin{lemma}[Mutual orthogonality]\cite{kaye2006introduction}
    Let $P_i\in \mathbb{C}^{m\times m}$ for $i\in [c]$ be orthogonal projectors such that they form a resolution of the identity. Then, $\{P_i\mid i \in [c]\}$ are mutually orthogonal, i.e.\! for distinct $i,j\in [c]$, it holds that $
        P_i P_j= 0.
    $
\end{lemma}

\section{Eigenvalue bound for  the distance quantum chromatic number: unified $\kappa$ bound } \label{section:kappabound}

The goal of this section is to extend the $\kappa$ bound from Theorem \ref{thm:hoffmanbetterbound}, to the distance-$k$ quantum setting of graph colorings. One source of motivation is the fact that the $\kappa$ bound has been shown to hold for the quantum chromatic number  \cite{wocjan2019spectral}.

\begin{theorem}\cite[Theorem 2]{wocjan2019spectral}\label{thm:quantumkappa}
    A connected graph $G$ on $n$ vertices with adjacency eigenvalues $\lambda_1\geq \cdots \geq \lambda_n$ and quantum chromatic number $\chi_q(G)$ satisfies
    \begin{equation}
    \chi_q(G)\geq 1+\kappa
    \end{equation}
    where $\kappa$ is the smallest integer such that 
    $$\lambda_1+\sum_{i=1}^{\kappa}\lambda_{n+1-i}\leq 0.$$

\end{theorem}

Another source of motivation comes from the diagram in Figure \ref{fig:diagramofparameters} in the Appendix. This gives an overview of variants of chromatic numbers in the literature, and how they compare to each other. For example, it is well-known that $\chi_q(G)\leq \chi(G)$, for any graph $G$. The orthogonal rank $\xi(G)$ on the other hand, is incomparable to the quantum chromatic number. That is, examples for both of the cases $\xi(G)<\chi_q(G)$ and $\xi(G)>\chi_q(G)$ have been constructed \cite{manvcinska2018oddities}. It is of great interest to know where in this diagram a certain lower bound is located. For example, if an eigenvalue bound for the chromatic number also happens to be a bound for the clique number, then it is clear that the bound can only attain a certain level of sharpness w.r.t.\! the chromatic number. From this point of view, an eigenvalue bound for the chromatic number should ideally not be a lower bound for any other variant of the chromatic number. On the other hand, if an eigenvalue bound for the chromatic number is also a lower bound for the quantum chromatic number, then we could infer the quantum chromatic number whenever the lower bound is sharp.

The $\kappa$ bound is a lower bound for the quantum chromatic number $\chi_q$, see Theorem \ref{thm:quantumkappa}, but not for the fractional chromatic number $\chi_f$. The latter can be shown by considering the Grötzsch graph, see \cite{mycielski1955coloriage}. These two considerations almost completely determine the position of the $\kappa$ bound in the diagram of parameters, with the exception of its relation to the orthogonal rank $\xi(G)$, which is an open problem \cite{wocjan2019spectral}. In extending bounds for the chromatic number to the distance-$k$ setting, we would also like the extensions to maintain their relative position in the diagram of parameters. In particular, a distance-$k$ extension of the $\kappa$ bound would ideally also be a lower bound for the distance-$k$ quantum chromatic number.

Next, we show that the $\kappa$ bound holds in the much more general distance-$k$ setting. To do so we need the following Theorem \ref{thm:pinchingdistancek}, which is the distance-$k$ analogue of \cite[Theorem 1]{wocjan2019spectral}. For our purposes the following formulation is the most useful, but in \cite{abiad2025eigenvalue} another analogous theorem can be found, that offers a more complete description of the pinching operation, which characterizes distance-$k$ quantum colorings.

\begin{theorem}\label{thm:pinchingdistancek}
    Let $\{P_{v,i}\mid  v\in V, i\in [c]\}$ be a distance-$k$ quantum $c$-coloring of $G$ in dimension $d$. Let $p\in\mathbb{R}^k$ be a polynomial of degree $k$. Then, the following block-diagonal orthogonal projectors
    \begin{equation}\label{eq:pinchingdistancek}
        P_i=\sum_{v\in V}e_ve_v^{\dagger} \otimes P_{v,i}\in \mathbb{C}^{n\times n}\otimes\mathbb{C}^{d\times d}
    \end{equation}
    define a pinching operation, and \begin{equation} \label{eq:pinchingpA}
        P_i(p(A)\otimes I_d)P_i=\sum_{v\in V}p(A)_{vv}(e_ve_v^{\dagger}\otimes P_{v,i}),
    \end{equation} for each $i\in [c]$.
    In particular, the projectors $\{P_i\mid i \in [c]\}$ form a resolution of the identity.
\end{theorem}
\begin{proof}
    We show that $\{P_i\mid i\in [c]\}$ form a resolution of the identity by noting that the orthogonal projectors $P_{v,i}$ satisfy completeness, since this condition holds for $G^k$ and it is independent of the edge set. We include the calculation for completeness;
    \begin{align*}
        \sum_{i\in[c]}P_i&=\sum_{i\in[c]}\sum_{v\in V}e_ve_v^{\dagger} \otimes P_{v,i}\\
        &=\sum_{v\in V}e_ve_v^{\dagger} \otimes \sum_{i\in[c]}P_{v,i} \\
        &=\sum_{v\in V}e_ve_v^{\dagger} \otimes I_d \\
        &=I_n \otimes I_d=I_{nd}.
    \end{align*}
    So we can define a pinching operation. 
    
    Next we show the second equality:
    \begin{align*}
        P_i(p(A)\otimes I_d)P_i &=\sum_{v,w\in V}(e_ve_v^{\dagger} \otimes P_{v,i})(p(A)\otimes I_d)(e_we_w^{\dagger} \otimes P_{w,i}) \\
        &=\sum_{v,w\in V}(e_ve_v^{\dagger}p(A))\otimes(P_{v,i})(e_we_w^{\dagger} \otimes P_{w,i}) \\
        &=\sum_{v,w\in V}(e_ve_v^{\dagger}p(A)e_we_w^{\dagger})\otimes(P_{v,i} P_{w,i}).
    \end{align*}
    If $v\neq w$, then either $d_G(v,w)\leq k$, in which case $P_{v,i} P_{w,i}=0$ by orthogonality, or $d_G(v,w)> k$, in which case $p(A)_{vw}=0$. Also, $e_v^{\dagger}p(A)e_w=p(A)_{vw}$, so
    \begin{align*}
        P_i(p(A)\otimes I_d)P_i &=\sum_{v,w\in V}(p(A)_{vw}e_ve_w^{\dagger})\otimes(P_{v,i} P_{w,i}) \\
        &=\sum_{v\in V}p(A)_{vv}(e_ve_v^{\dagger}\otimes P_{v,i} ).
 \qedhere   \end{align*}
\end{proof}

Now we are ready to prove our main result (Theorem \ref{thm:quantumhoffmanbetterbound}), which is the distance-$k$ analogue of Theorem \ref{thm:quantumkappa}, and improves the previously Hoffman-type bound for the (quantum) distance chromatic number, see \cite[Theorem 4.3]{abiad2022optimization} (and \cite[Theorem 14]{abiad2025eigenvalue}). To simplify notation, we use the following definition.

\begin{definition}
    Let $A\in \mathbb{R}^{n\times n}$ be a real-symmetric matrix, $p\in \mathbb{R}_k[x]$, and $\kappa$ be the smallest integer such that 
\begin{equation}\label{eq:kappathresholddef}
    (\kappa +1)W(p)\geq p(\lambda_1(A))+\sum_{i=1}^{\kappa}\lambda_{n-i+1}(p(A)),
\end{equation}
then we denote $\kappa$ as $\kappa(p(A))$ or $\kappa(p)$. 
\end{definition}

\begin{theorem}[Unified $\kappa$ bound]\label{thm:quantumhoffmanbetterbound}
Let $G=(V,E)$ be a connected graph on $n$ vertices and let $p\in \mathbb{R}_k[x]$ for $k \in \mathbb{N}^+$ such that $p(\lambda_1(A))\geq \Lambda(p(A))$. 
Then, the distance-$k$ quantum chromatic number satisfies
\begin{equation}\label{bound:quantumhoffmanrefinement}
    \chi_{kq}(G)\geq \kappa(p) + 1.
\end{equation}
\end{theorem}

\begin{proof}
The proof is analogous to the proofs of \cite[Theorem 4.3]{abiad2022optimization} and \cite[Theorem 2]{wocjan2019spectral}. 

Let $p'(x)\in \mathbb{R}_k[x]$ be any polynomial of degree at most $k$. Let $p(x)=p'(x)-W(p')$ for all $x\in \mathbb{R}$, and observe that inequality \eqref{eq:kappathresholddef} holds for $p$ if and only it holds for $p'$ for all $\kappa\in \mathbb{N}^+$. This is because $p(A)=p'(A)-W(p')I$, so $W(p)=0$, and $\lambda_i(p(A))=\lambda_i(p'(A))-W(p')$ for each $i\in [n]$. So it suffices to prove inequality \eqref{bound:quantumhoffmanrefinement} for $p$.

Let $\{P_{v,i}\mid v\in V,i \in[c\}$ be a distance-$k$ quantum $c$-coloring in dimension $d$. Then, let $\{P_i\mid i\in [c]\}$ be the block-diagonal projectors given by Eq.\! \eqref{eq:pinchingdistancek}.
Let $\bm{\nu}\in \mathbb{C}^n$ denote the Perron eigenvector of $A$ corresponding to $\lambda_1(A)$. Let $f_j\in\mathbb{C}^d$ for $j\in [d]$ be the standard basis vectors of $\mathbb{C}^d$. Let $s_1,\ldots, s_m$ be an orthonormal basis of the subspace $$
    \mathcal{S}=\langle P_i(\bm{\nu}\otimes f_j)\mid i\in [c], j\in [d]\rangle.
$$
We show that the dimension $m$ of $\mathcal{S}$ satisfies $$
    d\leq m\leq cd.
$$
First, note that the $d$ vectors $\bm{\nu}\otimes f_j$ for $j\in [d]$ are orthogonal since $(\bm{\nu}\otimes f_i)^{\dagger}(\bm{\nu}\otimes f_j)=(\bm{\nu}^{\dagger}\bm{\nu})\otimes (f_i^{\dagger}f_j)$. Furthermore, we have $$
    \bm{\nu}\otimes f_j=\sum_{i\in [c]}P_i(\bm{\nu}\otimes f_j),
$$
since the $P_i$ form a resolution of the identity, so $d\leq m$. We also have $m \leq cd$, since $\mathcal{S}$ is spanned by $cd$ vectors.

Let $S\in \mathbb{C}^{nd \times m}$ be the matrix with columns $s_1,\ldots,s_m$. We show that the matrix $S^{\dagger}(p(A)\otimes I_d)S$ has largest eigenvalue $p(\lambda_1(A))$ with multiplicity at least $d$. First, note that, by assumption, $p(A)$ has largest eigenvalue $p(\lambda_1(A))$ with eigenvector $\bm{\nu}$. There exist $d$ orthogonal vectors $y_1,\ldots,y_d\in \mathbb{C}^m$ such that $Sy_j=\bm{\nu}\otimes f_j$, since the vectors $\bm{\nu}\otimes f_j$ for $j\in[d]$ are contained in $\mathcal{S}$, the column space of $S$. We can calculate \begin{align*}
    S^{\dagger}(p(A)\otimes I_d)Sy_j&=S^{\dagger}(p(A)\otimes I_d)(\bm{\nu}\otimes f_j) \\
    &= S^{\dagger}(p(A)\bm{\nu})\otimes f_j \\
    &= p(\lambda_1(A))S^{\dagger}(\bm{\nu}\otimes f_j) \\
    &= p(\lambda_1(A))y_j.
\end{align*}
Secondly, by interlacing we have that each eigenvalue of $S^{\dagger}(p(A)\otimes I_d)S$ is at most $p(\lambda_1(A))$, which proves $S^{\dagger}(A\otimes I_d)S$ has largest eigenvalue $p(\lambda_1(A))$ with multiplicity at least $d$. 

We can always choose the basis vectors $s_1,\ldots,s_m$ such that there exists a $k_i \in [m]$ with $$
    P_{k_i}s_i=s_i \quad \text{and} \quad P_ks_i=0 \quad \text{for all}\quad k\neq k_i,
$$
for each $i$.
This is because the subspaces $\mathcal{S}_i=\langle P_i(\bm{\nu}\otimes f_j)\mid j\in [d]\rangle$ are mutually orthogonal and span $\mathcal{S}$, since the orthogonal projectors $\{P_i\mid i\in[c]\}$ are mutually orthogonal. We can thus choose each $s_i$ in a $\mathcal{S}_{k_i}$, and recall that $P_{k_i}^2=P_{k_i}$. This is where the proof most significantly differs from the proof in \cite{wocjan2019spectral}. Recall that by Theorem \ref{thm:pinchingdistancek} we have $$
    P_i(p(A)\otimes I_d)P_i=\sum_{v\in V}p(A)_{vv}(e_ve_v^{\dagger}\otimes P_{v,i})
$$
for each $i$. We now derive upper and lower bounds for the trace of $S^{\dagger}(p(A)\otimes I_d)S$. First, the last equality implies 
\begin{align*}
\tr(S^{\dagger}(p(A)\otimes I_d)S) &= \sum_{i=1}^ms_i^{\dagger}(p(A)\otimes I_d)s_i \\
&=\sum_{i=1}^ms_i^{\dagger}P_{k_i}(p(A)\otimes I_d)P_{k_i}s_i \\
&=\sum_{v\in V}p(A)_{vv}\sum_{i=1}^ms_i^{\dagger}(e_ve_v^{\dagger}\otimes P_{v,k_i})s_i.
\end{align*}
Now note that $e_ve_v^{\dagger}\otimes P_{v,k_i}$ is positive semidefinite for each $v\in V$ and $i\in[c]$, since it also an orthogonal projector. So, $s_i^{\dagger}(e_ve_v^{\dagger}\otimes P_{v,k_i})s_i\geq 0$ and since $W(p)=0$ we have $p(A)_{vv}\leq 0$ for each $v\in V$, so
\begin{align*}
\tr(S^{\dagger}(p(A)\otimes I_d)S) &\leq 0.
\end{align*}
Now for the lower bound we have 
\begin{align*}
    \tr(S^{\dagger}(p(A)\otimes I_d)S) &=\sum_{i=1}^m\lambda_{m-i+1}(S^{\dagger}(p(A)\otimes I_d)S) \\
    &=\sum_{i=1}^{m-d}\lambda_{m-i+1}(S^{\dagger}(p(A)\otimes I_d)S)+dp(\lambda_1(A)) \\
    &\geq\sum_{i=1}^{m-d}\lambda_{n-i+1}(p(A)\otimes I_d)+dp(\lambda_1(A)),
\end{align*}
where we use that $S^{\dagger}(p(A)\otimes I_d)S$ has eigenvalue $p(\lambda_1(A))$ with multiplicity $d$ and the last inequality is obtained by the interlacing Theorem \ref{thm:interlacing}. Let $\kappa_d$ be the smallest integer such that $$
    0\geq\sum_{i=1}^{\kappa_d}\lambda_{n-i+1}(p(A)\otimes I_d)+dp(\lambda_1(A)).
$$
Such a $\kappa_d$ exists, since $m-d$ suffices, i.e.\! $$
    0\geq\sum_{i=1}^{m-d}\lambda_{n-i+1}(p(A)\otimes I_d)+dp(\lambda_1(A)),
$$ by the above inequalities. Observe that $\kappa_d\leq m-d\leq cd-d=d(c-1)$. Furthermore, we have \begin{align*}
0&\geq\sum_{i=1}^{\kappa_d}\lambda_{n-i+1}(p(A)\otimes I_d)+dp(\lambda_1(A)) \\
&\geq\sum_{i=1}^{d\lceil\kappa_d/d\rceil}\lambda_{n-i+1}(p(A)\otimes I_d)+dp(\lambda_1(A)) \\
&=d\sum_{i=1}^{\lceil\kappa_d/d\rceil}\lambda_{n-i+1}(p(A))+dp(\lambda_1(A)),
\end{align*} 
where the second inequality takes advantage of the fact that all eigenvalues in the summation are negative. To see this, recall the minimality of $\kappa_d$, and the fact that the spectrum of $p(A) \otimes I_d$ is the spectrum of $p(A)$ with each multiplicity multiplied by $d$. Now we divide the last inequality by $d$ and obtain $$
    0\geq\sum_{i=1}^{\lceil\kappa_d/d\rceil}\lambda_{n-i+1}(p(A))+p(\lambda_1(A)).
$$ 
This shows there exists a minimal integer $\kappa$ such that 
$$
    0\geq\sum_{i=1}^{\kappa}\lambda_{n-i+1}(p(A))+p(\lambda_1(A)),
$$ 
and $\kappa \leq \lceil \kappa_d/d\rceil\leq \lceil d(c-1)/d\rceil=c-1$. Lastly recall from the comments made at the beginning of the proof that for any $\kappa'$ we have $$
    0\geq\sum_{i=1}^{\kappa'}\lambda_{n-i+1}(p(A))+p(\lambda_1(A)),
$$
if and only if $$
    (\kappa' +1)W(p')\geq\sum_{i=1}^{\kappa'}\lambda_{n-i+1}(p'(A))+p'(\lambda_1(A)).
$$

We conclude that $\kappa+1\leq c$ for any distance-$k$ quantum $c$-coloring and any polynomial $p$ of degree at most $k$, where $\kappa$ is minimal in inequality \eqref{eq:kappathresholddef}.
\end{proof}

\subsection{Optimization of the unified $\kappa$ bound (Theorem \ref{thm:quantumhoffmanbetterbound})}

Next we propose a Binary Integer Linear Program (BILP) to optimize the bound from Theorem \ref{thm:quantumhoffmanbetterbound}. We assume that $G$ is a connected graph on $n$ vertices and $p\in \mathbb{R}_k[x]$ such that $W(p)=0$. Such a $p$ that yields the optimal bound exists, since we can choose the constant term appropriately without changing the value of the bound.

We may reformulate the optimization problem as the maximization of $\kappa:=\sum_{i=2}^ne_i$ for $e_i\in\{0,1\}$ with some constraints. If we let $p(x)=a_kx^k+\ldots+a_0$, then the eigenvalues $p(\lambda_i)$ are linear in $a_0,\ldots,a_k$. Similarly, each $p(A)_{ii}$ is linear in $a_0,\ldots,a_k$, since we can calculate the diagonals of each $A^j$ in advance. To implement the constraint $W(p)=0$, we will assume $p(A)_{mm}=0$ for each $m \in [n]$ fixed, and run $n$ distinct BILP's and take the maximum. This is analogous to what was done for the Hoffman-type bound in \cite{abiad2024eigenvalue}. To implement some of the constraints using standard LP tricks we will also need a large constant $M\gg 0$, and a small $\varepsilon > 0$. 

{\small{
\setlength{\jot}{1em}
\begingroup
\allowdisplaybreaks

\begin{BILP}\label{BILP:kappak}

\begin{align*}
\textbf{Objective: }& &\\
    & \max \quad  \sum_{i=2}^{n} e_i  &\text{(note that this is } := \kappa \text{)}\\
\textbf{Constraints:} & &\\
    & \sum_{j=0}^ka_j\lambda_1^j +\sum_{i=2}^nz_i + t =0   & \constr[constr:kappakthreshold]\\
    &t+\varepsilon \leq -z_i+M(1-e_i)& \text{for } i\in [2,n] \constr[constr:kappakthresholdstrict]\\
    & z_i \leq \sum_{l=0}^ka_l\lambda_j^l - z_j + M(1-e_i)& \text{for } i,j\in [2,n] \constr[constr:kappakordering]\\
    & \sum_{j=0}^ka_j(A^j)_{mm}=0 & \constr[constr:kappakWp1]\\
    & \sum_{j=0}^ka_j(A^j)_{ii}\leq 0 & \text{for }  i \in [n]\setminus\{m\}\constr[constr:kappakWp2]\\
    & \sum_{j=0}^ka_j\lambda_i^j \leq z_i &\text{for } i\in [2,n] \constr[constr:kappakfixz1]\\
    & z_i \leq \sum_{j=0}^ka_j\lambda_i^j + M(1-e_i)& \text{for } i\in [2,n] \constr[constr:kappakfixz2]\\
    & -Me_i\leq z_i & \text{for } i\in [2,n]\constr[constr:kappakfixz3]
\end{align*}
\end{BILP}
\endgroup
}}

We summarize the variables of the BILP, and their interpretations as follows. 
\begin{itemize}
    \item $a_0,\ldots,a_k \in \mathbb{R}$: the $a_j$'s are the coefficients of  $p(x)=a_kx^k+\ldots+a_0$.
    \item $e_2,\ldots, e_n\in \{0,1\}$: for each $i$ we interpret $e_i=1$ as '$p(\lambda_i)$ is one of the $\kappa$ smallest eigenvalues, taking multiplicities into account,' and $e_i=0$ is interpreted as the negation of that. In other words, the BILP is modeled such that $\sum_{i=1}^\kappa \lambda_{n-i+1}(p(A))=\sum_{i=1}^{n-1}e_ip(\lambda_i)$.
    \item $z_2,\ldots,z_n \in \mathbb{R}$: for each $i$ the BILP is modeled such that $z_i=p(\lambda_i)$ if $e_i=1$, and $z_i=0$ otherwise. These variables are necessary to resolve the issues of non-linearity when multiplying $e_i$'s and $a_j$'s.
    \item $t \in \mathbb{R}_{\geq 0}$: this is a dummy variable which is unnecessary, but simplifies the formulation of the BILP.
\end{itemize}

We now give descriptive explanations for each of the constraints.
\begin{itemize}
    \item \constrref{constr:kappakthreshold} ensures that the condition for $\kappa$ given in Eq.\! \eqref{eq:kappathresholddef} holds.
    \item \constrref{constr:kappakthresholdstrict} ensures that $\sum_{i=2}^n e_i$ is the smallest $\kappa$ for which \constrref{constr:kappakthreshold} is satisfied.
    \item \constrref{constr:kappakordering} ensures that the $p(\lambda_i)$ for which $e_i=1$ are the smaller than all $p(\lambda_j)$ for which $e_j=0$, i.e.\! they are the $\kappa$ smallest $p(\lambda_i)$'s.
    \item \constrref{constr:kappakWp1} and \constrref{constr:kappakWp2} ensure that $W(p)=0$, recalling the assumption that $p(A)_{mm}=0$.
    \item \constrref{constr:kappakfixz1}, \constrref{constr:kappakfixz2} and \constrref{constr:kappakfixz3} ensure that $z_i=p(\lambda_i)$ if $e_i=1$ and $z_i=0$ if $e_i=0$.
\end{itemize}

\section{Consequences of Theorem \ref{thm:quantumhoffmanbetterbound}}\label{section:consequenceskappa}

\subsection{A \texorpdfstring{$\kappa$}{Kappa} bound for the classical distance-\texorpdfstring{$k$}{k} chromatic number} 

\begin{corollary}
\label{cor: newboundchikhoffmanstronger}
Let $G=(V,E)$ be a connected graph on $n$ vertices and let $p\in \mathbb{R}_k[x]$ for $k \in \mathbb{N}^+$ such that $p(\lambda_1(A))\geq \Lambda(p(A))$.
Then, the distance-$k$ chromatic number satisfies \begin{equation}
    \chi_{k}(G)\geq \kappa(p) + 1.
\end{equation}
\end{corollary}
\begin{proof}
    This is proven more generally in Theorem \ref{thm:quantumhoffmanbetterbound}, since $ \chi_{kq}(G) \geq 1 +\kappa(p)$ is shown and $\chi_k(G)\geq \chi_{kq}(G)$ holds.
\end{proof}

Corollary \ref{cor: newboundchikhoffmanstronger} is a strengthening of the known Hoffman-type bound \cite[Theorem 4.3]{abiad2022optimization}. Moreover, tightness of the Hoffman-type bound implies that the multiplicity of the smallest eigenvalue $\lambda(p)$ of $p(A)$ is at least $\chi_k-1$. It is valuable to note that almost all graphs have simple spectrum \cite{tao2017random} -- that is, almost all graphs only have adjacency eigenvalues of multiplicity $1$. This clearly impedes the Hoffman (for $k=1$) bound in particular from being sharp in almost all graphs. For $k=2$, we observe experimentally that the $k=2$ unified $\kappa$ bound outperforms the Hoffman-type bound with increasing proportion for larger graphs, see Figure \ref{fig:randomgraphsk2}.

\begin{proposition}\label{prop:kappabetterthanhoffman}
    Let $G$ be a non-empty graph on $n$ vertices. Let $p\in \mathbb{R}_k[x]$ for $k \in \mathbb{N}^+$ such that $p(\lambda_1)\geq \Lambda(p(A))$. 
    Then, we have $$\frac{p(\lambda_1)-\lambda(p)}{W(p)-\lambda(p)}\leq1+\kappa(p)\leq \chi_k(G),$$
    and if there is equality everywhere, then the multiplicity of $\lambda(p)$ as an eigenvalue of $p(A)$ is at least $\chi_k-1$. 
\end{proposition}
\begin{proof}
First, note that \begin{align*}
    p(\lambda_1)+\kappa\lambda(p)\leq p(\lambda_1)+\sum_{i=1}^{\kappa}\lambda_{n+1-i}(p(A))\leq (\kappa+1)W(p), 
\end{align*} so we have $$
    p(\lambda_1)-\lambda(p)\leq (\kappa+1)(W(p)-\lambda(p)).
$$
And, $W(p)>\lambda(p)$ follows by considering the trace of $p(A)$, which proves the result. Moreover, equality is only possible if $\lambda(p)=\lambda_{n+1-i}(p(A))$ for each $i \in \{1,\ldots,n\}$. 
\end{proof}

\subsection{Tightness of the unified $\kappa$ bound to compute the quantum distance chromatic parameter}

We consider some examples that show how the new unified $\kappa$ bound from Theorem \ref{thm:quantumhoffmanbetterbound} improves on the known Hoffman-type bound, as shown in Proposition \ref{prop:kappabetterthanhoffman}. For example, we can find infinite classes of graphs where the new unified $\kappa$ bound is strictly better than the Hoffman-type bound.
\begin{example}
    It is easily shown that the Hoffman-type bound for $k=2$, see \cite[Theorem 3]{abiad2025eigenvalue}, is sharp for all $K_{n,n}$. However, when we remove an edge $e$ from $K_{n,n}$, the Hoffman-type bound is never sharp for $n\geq 3$. The unified $\kappa$ bound however, always attains equality with $\chi_2$ for the graphs $K_{n,n} -e$ where $n\geq 3$. We can show this explicitly using the optimal polynomial $p_2$ for the Hoffman-type bound to calculate $\kappa(p_2)$, see \cite[Theorem 3]{abiad2025eigenvalue}.
\end{example}

Note that the unified $\kappa$ bound in the distance-$k$ case is also sharp for numerous graphs for which the Hoffman-type bound from \cite[Theorem 4.3]{abiad2022optimization} is sharp. For example, this includes the Lee graphs and hypercubes that were studied in \cite{abiad2024eigenvalue}.

We can apply Theorem \ref{thm:quantumhoffmanbetterbound} to find the distance-$k$ quantum chromatic number of new graphs. Indeed, equality in $1+\kappa =\chi_k(G)$ implies $1+\kappa =\chi_{kq}(G)$, since $1+\kappa\leq \chi_{kq}(G)\leq \chi_k(G)$. Next we provide an example that also has the interesting property that the distance-$1$ $\kappa$ bound  calculated directly for $G^2$ does not attain $\chi_2(G)$, but the distance-$2$ $\kappa$ bound for $G$, does attain $\chi_2(G)$.

\begin{example}
The Truncated Prism graph $G$ shown in Figure \ref{fig:Truncated_Prism} has distance-$2$ chromatic number $\chi_2(G)=5$, which was determined by computation. Moreover, applying the $\kappa$ bound with $$
p(x)=x^2-\frac{1-\sqrt{13}}2x
$$ yields $1+\kappa=5=\chi_2(G)$. This polynomial was found using BILP \ref{BILP:kappak}, but note that $\frac{1-\sqrt{13}}2$ is an eigenvalue of the graph. Interestingly, if we were to apply the distance-$1$ $\kappa$ bound from Theorem \ref{thm:hoffmanbetterbound} to the power graph $G^2$, we would only obtain a lower bound of $4$ for the chromatic number of $G^2$. This does not enable us to determine the quantum chromatic number of $G^2$, while the distance-$2$ $\kappa$ bound does yield $\chi_q(G^2)=5$.

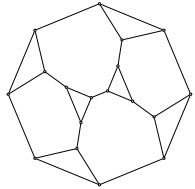
\begin{figure}[!htb]
        \centering
\scalebox{.3}{
        \begin{tikzpicture}[main_node/.style={circle,draw,minimum size=1pt,inner sep=1pt]}]

\node[main_node] (0) at (3.1428571428571423, 0.8571428571428563) {};
\node[main_node] (1) at (1.9716796384634696, 3.685965352749184) {};
\node[main_node] (2) at (1.538639216670851, -0.14391162440371552) {};
\node[main_node] (3) at (1.9716796384634696, -1.9716796384634696) {};
\node[main_node] (4) at (-3.685965352749184, -1.9716796384634696) {};
\node[main_node] (5) at (-1.8483555109214156, -1.5414511674617124) {};
\node[main_node] (6) at (-1.212854632186795, 0.7081094652272153) {};
\node[main_node] (7) at (-2.3123273914135067, 1.1692693949284454) {};
\node[main_node] (8) at (-1.6683906603063017, -0.3913632939994969) {};
\node[main_node] (9) at (-0.5014310820989203, 1.0033642982676376) {};
\node[main_node] (10) at (-3.2529249309565653, 1.8567913632939987) {};
\node[main_node] (11) at (-4.857142857142858, 0.8571428571428563) {};
\node[main_node] (12) at (-0.04730102937484304, 2.1070549836806416) {};
\node[main_node] (13) at (-3.685965352749184, 3.685965352749184) {};
\node[main_node] (14) at (-0.8571428571428563, 4.857142857142858) {};
\node[main_node] (15) at (0.13125784584484101, 3.2585488325382883) {};
\node[main_node] (16) at (-0.8571428571428563, -3.1428571428571423) {};
\node[main_node] (17) at (0.5994476525232244, 0.545016319357269) {};

 \path[draw, thick]
(0) edge node {} (1) 
(0) edge node {} (2) 
(0) edge node {} (3) 
(1) edge node {} (14) 
(1) edge node {} (15) 
(2) edge node {} (3) 
(2) edge node {} (17) 
(3) edge node {} (16) 
(4) edge node {} (5) 
(4) edge node {} (11) 
(4) edge node {} (16) 
(5) edge node {} (8) 
(5) edge node {} (16) 
(6) edge node {} (7) 
(6) edge node {} (8) 
(6) edge node {} (9) 
(7) edge node {} (8) 
(7) edge node {} (10) 
(9) edge node {} (12) 
(9) edge node {} (17) 
(10) edge node {} (11) 
(10) edge node {} (13) 
(11) edge node {} (13) 
(12) edge node {} (15) 
(12) edge node {} (17) 
(13) edge node {} (14) 
(14) edge node {} (15) 
;

\end{tikzpicture}}
        \caption{The Truncated Prism graph has distance-$2$ chromatic number $\chi_2=5$, which equals its distance-$2$ quantum chromatic number. }
        \label{fig:Truncated_Prism}
    \end{figure}
\end{example}

\section{Eigenvalue bound for the distance vector chromatic number}\label{chapter:vector}

In this section, we show new sharp bounds for the distance-$k$ vector chromatic number. These bounds extend the Hoffman-type bound for the distance-$k$ chromatic number \cite[Theorem 4.3]{abiad2022optimization}, and the Hoffman bound on the vector chromatic number \cite[Theorem 1]{bilu2006tales}.

The \emph{vector chromatic number} of a graph $G=(V,E)$ is the smallest $c\in\mathbb{R}_{\geq 2}$ such that there exist unit vectors $u_v\in \mathbb{R}^n$ for $v\in V$ such that $$
    \langle u_v,u_w\rangle \leq \frac{-1}{c-1}
$$
for all adjacent $v,w\in V$ \cite{wocjan2018more}. We denote the vector chromatic number $c$ by $\chi_v(G)$. We define the \emph{distance-$k$ vector chromatic number} as the vector chromatic number of $G^k$ and denote it by $\chi_{kv}(G)=\chi_v(G^k)$. Vector chromatic numbers were introduced in \cite{karger1998approximate}, and are a semidefinite relaxation of the chromatic number. 

Note that we have $\chi_v(G)\leq \chi(G)$ for all graphs $G$ \cite{karger1998approximate}, and it was shown that there exists families of graphs with vector chromatic numbers that are much smaller than their chromatic number \cite{feige2004graphs}.

The following theorem is an important generalization of the Hoffman bound for the vector chromatic number. 

\begin{theorem}\cite[Theorem 1]{bilu2006tales}\label{thm:biluvector}
    Let $G$ be a graph on $n$ vertices with adjacency eigenvalues $\lambda_1\geq \cdots\geq\lambda_n$. Then, $$
        \chi_v(G)\geq 1-\frac{\lambda_1}{\lambda_n}.
    $$
\end{theorem}
Motivated by the above, in this section we show that the Hoffman-type bound \cite[Theorem 4.3]{abiad2022optimization} also holds for the vector chromatic number, see Corollary \ref{thm:vectorchromW}. This also completely determines where we can place the Hoffman-type bound in the diagram in Figure \ref{fig:diagramofparameters} (see Appendix), i.e.\! in between the clique number $\omega$ and the vector chromatic number $\chi_v$.

In order to prove our next main result (Theorem \ref{thm:vectorchromR}) we will use the following known lemma.

\begin{lemma}\cite[Lemma 4]{bilu2006tales}\label{lem:bilu}
    Let $A=(a_{ij})\in\mathbb{R}^{n\times n}$ be a real symmetric matrix with smallest eigenvalue $\lambda_n$, then $$
        \lambda_n = \min_{v_1,\ldots,v_n\in \mathbb{R}^n}\frac{\sum_{i,j=1}^na_{ij}\langle v_i,v_j\rangle}{\sum_{i=1}^n\langle v_i,v_i\rangle}.
    $$
\end{lemma}
\subsection{Hoffman-type bound on the distance vector chromatic number}

For a polynomial $p \in \mathbb{R}_k[x]$ and connected graph $G$ with Perron eigenvector $\bm{\nu}$, we define the parameter
\begin{align*}R(p) &=\frac{\sum_{v\in V}p(A)_{vv}\bm{\nu}_v^2}{\sum_{v\in V}\bm{\nu}_v^2}.
\end{align*}
If $G$ is regular, then $R(p)$ reduces to the average of $p(A)_{vv}$ taken over $V$. If we normalize $\bm{\nu}$ such that $\Vert \bm{\nu}\Vert_2=1$, then $R(p)$ is essentially the average of the diagonal $p(A)$ weighted by the squares of the corresponding elements of $\bm{\nu}$.
First, we prove a lower bound for this parameter, which will be useful in the proof of Theorem \ref{thm:vectorchromR}.

\begin{lemma}\label{lem:vectorchromedge}
    Let $G=(V,E)$ be a connected graph and $p\in \mathbb{R}_k[x]$, then $$
        \lambda(p)\leq R(p).
    $$
\end{lemma}
\begin{proof}
    Let $\bm{\nu}$ be the Perron eigenvector of $G$ indexed by the vertices of $G$. Let $u_v=\bm{\nu}_v\cdot e_v$ for each $v\in V$, where $\{e_v\}_{v\in V}$ are the canonical basis vectors of $\mathbb{R}^n$. By Lemma \ref{lem:bilu} we have $$
        \lambda(p)\leq \frac{\sum_{v,w\in V}p(A)_{vw}\langle u_v, u_w\rangle}{\sum_{v\in V}\langle u_v, u_v\rangle}=R(p),
    $$
    since $\langle u_v, u_w\rangle = 0$ whenever $v\neq w$, and $\langle u_v, u_v\rangle=\bm{\nu}_v^2$ for all $v,w\in V$.
\end{proof}

Now we are ready to prove the main result of this section.

\begin{theorem}\label{thm:vectorchromR}
    Let $G$ be a non-empty connected graph with $n$ vertices and adjacency matrix $A$ having eigenvalues $\lambda_1\geq\cdots\geq \lambda_n$, and Perron eigenvector $\bm{\nu}$. Let $p \in \mathbb{R}_k[x]$ with corresponding parameters $\lambda(p):= \min_{i\in[2,n]} \{p(\lambda_i)\}$, $R(p)=\frac{\sum_{v\in V}p(A)_{vv}\bm{\nu}_v^2}{\sum_{v\in V}\bm{\nu}_v^2}$, and assume $R(p)\neq \lambda(p)$. Then, the distance-$k$ vector chromatic $\chi_{kv}(G)$ number satisfies
    \begin{equation*}
        \chi_{kv}(G) \geq \frac{p(\lambda_1) - \lambda(p)}{R(p) - \lambda(p)}.
    \end{equation*}
\end{theorem}
\begin{proof}
    Let $u_v\in \mathbb{R}^n$ for $v\in V$ be the unit vectors on which the distance-$k$ vector chromatic number of $G$ is attained. Let $\chi_{kv}=\chi_{kv}(G)$. Let $\bm{\nu}$ denote the Perron eigenvector of $A$ corresponding to $\lambda_1$, with components indexed by $V$. By Lemma \ref{lem:bilu} applied to $p(A)$ and the vectors $\bm{\nu}_v\cdot u_v$ for $v\in V$ we have \begin{align*}
        \lambda(p)&\leq \frac{\sum_{v,w\in V}p(A)_{vw}\bm{\nu}_v\bm{\nu}_w\langle u_v,u_w\rangle}{\sum_{v\in V}\bm{\nu}_v^2\langle u_v,u_v\rangle}.
    \end{align*}
    Since $\sum_{v\in V}\bm{\nu}_v^2\langle u_v,u_v\rangle = \langle \bm{\nu},\bm{\nu}\rangle$, this inequality is equivalent to $$
        \lambda(p)\langle \bm{\nu},\bm{\nu}\rangle\leq \sum_{v\in V}p(A)_{vv}\bm{\nu}_v^2+\sum_{v\neq w\in V}p(A)_{vw}\bm{\nu}_v\bm{\nu}_w\langle u_v,u_w\rangle,
    $$ where we expanded the sum. Now recall that for a polynomial of degree at most $k$ we have $p(A)_{vw}=0$ if $d_G(v,w)>k$ for any $v,w\in V$. If $v,w$ are distinct and $d_G(v,w)\leq k$, then we have $\langle u_v,u_w\rangle\leq \frac{-1}{\chi_{kv}-1}$. This yields \begin{align*}
        \lambda(p)\langle \bm{\nu},\bm{\nu}\rangle &\leq\sum_{v\in V}p(A)_{vv}\bm{\nu}_v^2-\frac{1}{\chi_{kv}-1}\sum_{v\neq w\in V}p(A)_{vw}\bm{\nu}_v\bm{\nu}_w \\
        &=\left(1+\frac{1}{\chi_{kv}-1}\right)\sum_{v\in V}p(A)_{vv}\bm{\nu}_v^2 -\frac{1}{\chi_{kv}-1}\sum_{v,w\in V}p(A)_{vw}\bm{\nu}_v\bm{\nu}_w,
    \end{align*} by sum arithmetic. By the definition of $R(p)$ we have $$
    \lambda(p)\langle \bm{\nu},\bm{\nu}\rangle \leq
        \frac{\chi_{kv}}{\chi_{kv}-1}R(p)\langle \bm{\nu},\bm{\nu}\rangle -\frac{1}{\chi_{kv}-1}p(\lambda_1)\langle \bm{\nu},\bm{\nu}\rangle,
    $$ and dividing by $\langle \bm{\nu},\bm{\nu}\rangle>0$ yields $$
        \lambda(p)(\chi_{kv}-1)\leq \chi_{kv}R(p)-p(\lambda_1),
    $$ or $$
        \chi_{kv}(\lambda(p)-R(p))\leq \lambda(p) -p(\lambda_1).
    $$
    By Lemma \ref{lem:vectorchromedge}, and $\lambda(p)\neq R(p)$ we have $\lambda(p)-R(p)<0$, which concludes the proof.
\end{proof}
Note that the following corollary is exactly the Hoffman-type bound from \cite[Theorem 4.3]{abiad2022optimization}, but for the vector chromatic number. This helps in placing the Hoffman-type bound in the bounds diagram in Figure \ref{fig:diagramofparameters} (see Appendix).
\begin{corollary}\label{thm:vectorchromW}
    Let $G$ be a non-empty connected graph with $n$ vertices and eigenvalues $\lambda_1,\ldots,\lambda_n$ with adjacency matrix $A$. Let $p \in \mathbb{R}_k[x]$ with corresponding parameters $W(p) := \max_{u \in V}$ $\{p(A)_{uu}\}$ and $ \lambda(p):= \min_{i\in[2,n]} \{p(\lambda_i)\}$, and assume $W(p)\neq \lambda(p)$. Then, the distance-$k$ vector chromatic number satisfies 
        \begin{equation}
            \chi_{kv}(G) \geq \frac{p(\lambda_1) - \lambda(p)}{W(p) - \lambda(p)}.
        \end{equation}
\end{corollary}
\begin{proof}
Note that $W(p)\geq R(p)$, and therefore $R(p)> \lambda(p)$. This concludes the proof.
\end{proof}

\subsection{Optimization of Theorem \ref{thm:vectorchromR}}
We present a Linear Program (\LPref{LP:Rpbound}) to derive the optimal polynomial for Theorem \ref{thm:vectorchromR}, which is analogous to the LP derived in \cite{abiad2024eigenvalue} for the Hoffman-type bound \cite[Theorem 4.3]{abiad2022optimization}. 

Suppose we have a connected graph on $n$ vertices $G=(V,E)$ with adjacency eigenvalues $\spec(G)=\{\lambda_1\geq \cdots\geq \lambda_n\}$, $\ev(G)=\{\theta_0,\ldots,\theta_d\}$, and Perron eigenvector $\bm{\nu}$ such that $\Vert \bm{\nu} \Vert = 1$. We run the following LP once for each assumption $p(\lambda_{m})=p(\theta_{m'})=\lambda(p)$ for $m'\in [d]$. 

The only parameters of the LP \ref{LP:Rpbound} are $a_0,\ldots,a_k \in \mathbb{R}$, the coefficients of $p(x)=a_kx^k+\ldots+a_0$. The constraints are summarized as follows.
\begin{itemize}
    \item \constrref{constr:Rbounddenominator} ensures that the denominator is equal to $1$.
    \item \constrref{constr:Rboundminimallambda} ensures that $\lambda(p)=p(\lambda_{m'})$ for a fixed $m' \in [d]$.
\end{itemize}

{\small{
\begin{LP}\label{LP:Rpbound}
\begin{align*}
\textbf{Objective: }& &\\
    & \max \quad  \sum_{j=0}^{k} a_j\lambda_1^j -\sum_{j=0}^{k}a_j\lambda_m^j&\\
\textbf{Constraints:} & \\
    & \sum_{j=0}^k a_j \sum_{i = 1}^n(A^j)_{ii} \bm{\nu}^2_i-a_j\lambda_m^j=1&\constr[constr:Rbounddenominator]\\
    &\sum_{j=0}^{k} a_j\lambda_i^j -a_j\lambda_m^j\geq 0& \text{for } i\in [d] \constr[constr:Rboundminimallambda]\\
\end{align*}
\end{LP}
}}

\subsection*{Acknowledgments}

Aida Abiad is supported by NWO (Dutch Research Council) through the grant \linebreak VI.Vidi.213.085. The authors thank Clive Elphick for pointing out some related references.

\printbibliography

\newpage
\section{Appendix}
\subsection{Diagram of graph parameters}\label{ssec:parameterdiagram}
The diagram in Figure \ref{fig:diagramofparameters} represents inequalities between relevant graph parameters of any graph $G$. The original diagram, along with references and justifications for the inequalities, can be found in \cite{wocjan2018spectral}. Note that the $k$-distance analogous parameters, i.e.\! the respective parameters of the power graphs, also adhere to these inequalities. 

For clarity, we define exactly what it means when there is an arrow between two parameters, and what it may mean if there is no arrow. Suppose we have parameters $\beta_1,\beta_2$, i.e.\! maps from graphs to $\mathbb{R}$. Then, we have $\beta_1\longrightarrow\beta_2$ if, for all graphs $G$ we have $\beta_1(G)\leq \beta_2(G)$. Suppose that there is no arrow between two parameters, then there are multiple possibilities. It is possible that they are incomparable, which means that there exists a graph $G_1$ such that $\beta_1(G_1)<\beta_2(G_1)$ and a graph $G_2$ such that $\beta_1(G_2)>\beta_2(G_2)$. For example, $\xi$ and $\chi_q$ are incomparable \cite{manvcinska2018oddities}. However, it may also mean that it is unknown whether or not those parameters are comparable. Note that the diagram loops around from the right side to the left side of the page, and that a red cross emphasizes that a certain relation does not hold.

The following parameters are included in the diagram;
\begin{itemize}
    \item $\omega(G)$, the clique number;
    \item $1-\lambda_1/\lambda_n$, the Hoffman bound cf.~\cite{haemers2021hoffman};
    \item $\chi_v(G)$, the vector chromatic number;
    \item $\theta'(G)$, the Schrijver variant of the Lovász number;
    \item $\chi_{sv}(G)$, the strict vector chromatic number;
    \item $\theta(G)$, the Lovász number;
    \item $\theta^+(G)$, the Szegedy variant of the Lovász number;
    \item $\chi_{\text{vect}}(G)$, the vectorial chromatic number (\emph{not} the vector chromatic number);
    \item $\xi_f(G)$, the projective orthogonal rank;
    \item $\kappa + 1$, the $\kappa$ bound \cite{hoffman2003eigenvalues};
    \item $\xi(G)$, the orthogonal rank;
    \item $\chi_q(G)$, the quantum chromatic number;
    \item $\chi_f(G)$, the fractional chromatic number;
    \item $\chi_q^{(1)}(G)$, the rank-1 quantum chromatic number;
    \item $\xi'(G)$, the normalized orthogonal rank;
    \item $\chi_c(G)$, the circular chromatic number;
    \item $\chi(G)$, the standard chromatic number.
\end{itemize}

\begin{figure}[!h]
    \centering
\begingroup
\begin{align*}
\centering
\scalebox{1}{
\begin{tikzpicture}
\node at (0,0) (n1) {$\omega(G)$};
\node at (0, -2) (n5) {$1-\lambda_1/\lambda_n$};
\node at (3,0) (n2) {$\chi_v(G)=\theta'(\bar{G})$};
\node at (7,0) (n3) {$\chi_{sv}(G)=\theta(\bar{G})$};
\node at (10,0) (n4) {$\theta^+(\bar{G})$};
\path[->]  (n1) edge (n2);
\path[->]  (n2) edge (n3);
\path[->]  (n3) edge (n4);
\path[->]  (n5) edge (n2);
\end{tikzpicture}
}
\end{align*}
\begin{align*}
\centering
\scalebox{1}{
\begin{tikzpicture}
\node at (0,0) (n1) {$\theta^+(\bar{G})$};
\node at (3,0) (n2) {$\xi_f(G)$};
\node at (3,-2) (n7) {$\kappa+1$};
\node at (3,2) (n3) {$\lceil\theta^+
(\bar{G})\rceil=\chi_{\text{vect}}(G)$};
\node at (7,-2) (n4) {$\chi_f(G)$};
\node at (7,0) (n5) {$\chi_q(G)$};
\node at (7,2) (n6) {$\xi(G)$};
\path[->]  (n1) edge (n2);
\path[->]  (n1) edge (n3);
\path[->]  (n2) edge (n4);
\path[->]  (n2) edge (n5);
\path[->]  (n3) edge (n5);
\path[->]  (n2) edge (n6);
\path[->]  (n3) edge (n6);
\path[->]  (n7) edge (n5);
\path[->]  (n7) edge (n4);
\draw ($(n7)!0.5!(n4)$) node[red] {\Large $\times$};
\end{tikzpicture}
}
\end{align*}

\begin{align*}
\centering
\scalebox{1}{
\begin{tikzpicture}
\node at (0,-2) (n1) {$\chi_f(G)$};
\node at (0,0) (n2) {$\chi_q(G)$};
\node at (0,2) (n3) {$\xi(G)$};
\node at (3,0) (n4) {$\chi_q^{(1)}(G)$};
\node at (6,0) (n5) {$\xi'(G)$};
\node at (6,-2) (n6) {$\chi_c(G)$};
\node at (10,0) (n7) {$\lceil\chi_c(G)\rceil=\chi(G)$};
\path[->]  (n1) edge (n6);
\path[->]  (n2) edge (n4);
\path[->]  (n3) edge (n4);
\path[->]  (n4) edge (n5);
\path[->]  (n5) edge (n7);
\path[->]  (n6) edge (n7);
\end{tikzpicture}
}
\end{align*}
\endgroup
\caption{Diagram of graph parameters. Note that the relations loop from left to right. }
\label{fig:diagramofparameters}
\end{figure}

\clearpage
\subsection{Computational experiments for Sage named graphs}\label{ssec:sagenamedgraphs}

\begin{table}[!htp]
    \centering
    {\scriptsize{
    \begin{tabular}{|l|c|c|c|c|}
        \hline
        Name &  \cite[Theorem 4.3]{abiad2022optimization} & Theorem \ref{thm:vectorchromR} & Unified $\kappa$ bound & $\chi_2$ \\
        \hline
        \textbf{Bidiakis Cube} & 4 & 4 & 5 & 6 \\
\textbf{Blanusa First Snark Graph} & 4 & 4 & 5 & 6 \\
Blanusa Second Snark Graph & 5 & 5 & 5 & 6 \\
\textbf{Bull Graph} & 3 & 4 & 3 & 4 \\
\textbf{Butterfly Graph} & 3 & 4 & 3 & 5 \\
\textbf{Chvatal Graph} & 5 & 5 & 10 & 12 \\
\textbf{Claw Graph} & 2 & 3 & 2 & 4 \\
Clebsch Graph & 16 & 16 & 16 & 16 \\
\textbf{Dart Graph} & 3 & 4 & 3 & 5 \\
Desargues Graph & 4 & 4 & 4 & 6 \\
Diamond Graph & 4 & 4 & 4 & 4 \\
Dodecahedral Graph & 5 & 5 & 5 & 5 \\
Durer Graph & 5 & 5 & 5 & 6 \\
\textbf{Errera Graph} & 6 & 7 & 7 & 9 \\
\textbf{Flower Snark} & 4 & 4 & 5 & 6 \\
Folkman Graph & 7 & 7 & 7 & 10 \\
Fork Graph & 3 & 3 & 3 & 4 \\
Franklin Graph & 4 & 4 & 4 & 6 \\
\textbf{Frucht Graph} & 4 & 4 & 5 & 6 \\
Gem Graph & 4 & 4 & 4 & 5 \\
\textbf{Goldner Harary Graph} & 6 & 7 & 6 & 11 \\
\textbf{Golomb Graph} & 4 & 5 & 4 & 7 \\
\textbf{Grotzsch Graph} & 5 & 6 & 5 & 11 \\
Heawood Graph & 7 & 7 & 7 & 7 \\
\textbf{Herschel Graph} & 4 & 5 & 4 & 6 \\
Hexahedral Graph & 4 & 4 & 4 & 4 \\
Hoffman Graph & 6 & 6 & 6 & 8 \\
House Graph & 4 & 4 & 4 & 5 \\
\textbf{HouseX Graph} & 4 & 5 & 4 & 5 \\
Icosahedral Graph & 6 & 6 & 6 & 6 \\
\textbf{Krackhardt Kite Graph} & 5 & 6 & 5 & 8 \\
Moebius Kantor Graph & 4 & 4 & 4 & 4 \\
\textbf{Moser Spindle }& 4 & 5 & 4 & 7 \\
Octahedral Graph & 6 & 6 & 6 & 6 \\
Pappus Graph & 5 & 5 & 5 & 6 \\
Petersen Graph & 10 & 10 & 10 & 10 \\
\textbf{Poussin Graph} & 6 & 7 & 6 & 9 \\
\textbf{Robertson Graph} & 7 & 7 & 8 & 9 \\
Shrikhande Graph & 16 & 16 & 16 & 16 \\
\textbf{Sousselier Graph} & 4 & 6 & 4 & 8 \\
Thomsen Graph & 6 & 6 & 6 & 6 \\
Tietze Graph & 4 & 4 & 4 & 7 \\
Truncated Tetrahedral Graph & 4 & 4 & 4 & 4 \\
\textbf{Wagner Graph} & 4 & 4 & 8 & 8 \\
        \hline
    \end{tabular}
    }}
    \caption{The named Sage graphs with the Hoffman-type bound from \cite[Theorem 4.3]{abiad2022optimization} rounded up, the bound from Theorem \ref{thm:vectorchromR} rounded up, and the unified $\kappa$ bound from Theorem \ref{thm:quantumhoffmanbetterbound} optimized using \BILPref{BILP:kappak}, as well as the actual distance-$2$ chromatic number $\chi_2$. Graphs where the unified unified $\kappa$ bound or the bound from Theorem \ref{thm:quantumhoffmanbetterbound} improved on the known Hoffman-type bound.}
    \label{tab:sagenamedgraphs}
\end{table}

\newpage
\subsection{Computational experiments for small graphs}\label{ssec:smallgraphs}
\begin{table}[!htp]
    \centering
    \begin{tabular}{|c|c|c|c|c|}
        \hline
        $n$ & Hoffman bound & $\kappa$ bound & Total $\#$ graphs\\
        \hline
        3 & 2 & 2 & 2 \\
        4 & 4 & 6 & 6 \\
        5 & 6 & 21 & 21 \\
        6 & 20 & 111 & 112 \\
        7 & 45 & 780 & 853 \\
        8 & 184 & 8630 & 11117 \\
        \hline
    \end{tabular}
    \caption{Number of connected graphs on $n$ vertices such that the Hoffman bound is sharp for $\chi(G)$. And, the number of graphs for which the $\kappa$ bound Eq.\! \eqref{bound:hoffmanrefinement} is sharp. The total number of connected graphs on $n$ vertices is given in the last column. Computed by brute-forcing all graphs on at most $8$ vertices.}
    \label{tab:comparisonchromaticbounds}
\end{table}

\begin{figure}[!htp]
\centering
\hspace*{-2.5cm}
\begin{tikzpicture}
\begin{axis}[
    scale only axis,
    ybar,
    xlabel={$n$},
    ylabel={Estimated proportion},
    ylabel style={
      at={(axis description cs:0.05,0.5)},
      anchor=south,
      rotate=0
    },
    ymin=0,
    width=0.8\textwidth,
    height=7cm,
    bar width=8pt,
    xtick={5,...,23},
    enlarge x limits=0.05
]

\addplot+[
    ybar,
    fill=gray!60,
    draw=gray!80
] table[
    x=n,
    y=X
] {
n   X
5   0.028
6   0.044
7   0.058
8   0.077
9   0.110
10  0.119
11  0.139
12  0.179
13  0.197
14  0.218
15  0.245
16  0.266
17  0.287
18  0.308
19  0.328
20  0.365

};

\end{axis}
\end{tikzpicture}
\caption{For each $n$, $1000$ Erdős–Rényi random graphs on $n$ vertices with $p=1/2$ are generated using SageMath (random seed $100$). Then, for each graph, the optimal polynomial $p_2$ for the $k=2$ Hoffman-type bound is computed. Both the Hoffman-type bound and the unified $\kappa$ bound are then computed using $p_2$. Whenever the unified $\kappa$ bound is strictly larger than the ceiling Hoffman-type bound, one instance of strict out-performance is recorded. The proportions of strict out-performance are shown above. Note that the unified $\kappa$ bound is computed on $p_2$, and not on its own optimal polynomial solely to reduce the computation time. }
\label{fig:randomgraphsk2}
\end{figure}
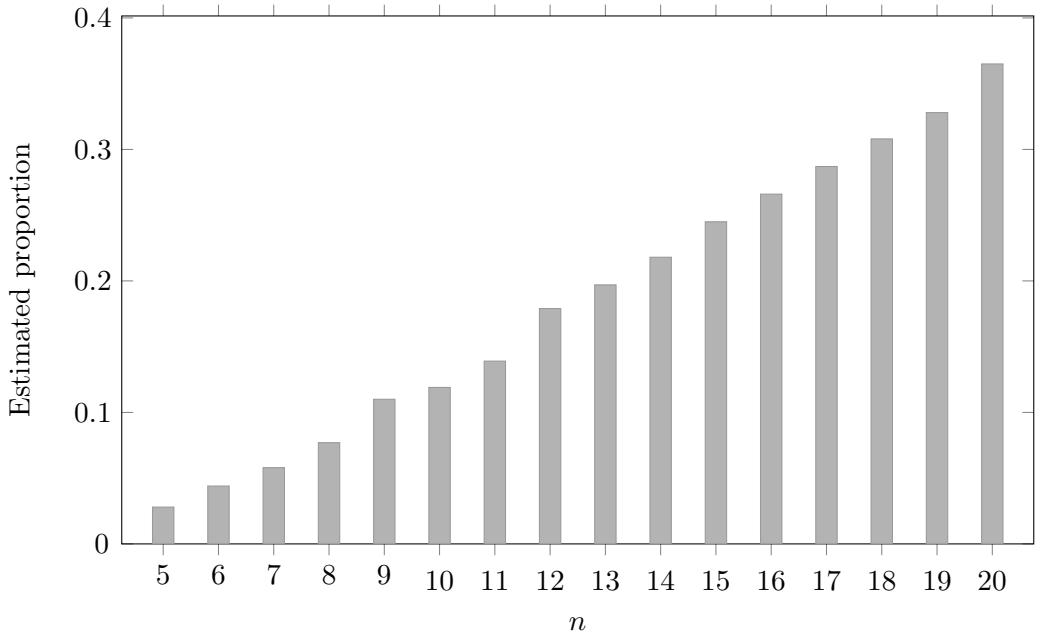

\end{document}